\documentclass[10pt,a4paper]{article}

\usepackage{vmargin}
\setmarginsrb{3cm}{1cm}{3cm}{2.5cm}{1cm}{1cm}{1cm}{1cm}

\usepackage[latin1]{inputenc}
\usepackage[T1]{fontenc}
\usepackage{color}
\usepackage{amsmath}
\usepackage{amsthm}
\usepackage{amsfonts}
\usepackage{amssymb}
\usepackage{graphicx}

\newtheorem{theorem}{Theorem}[section]
\newtheorem{proposition}[theorem]{Proposition}

\newtheorem{corollary}[theorem]{Corollary}
\newtheorem{definition}[theorem]{Definition}

\newtheorem{remark}[theorem]{Remark}

\def\R{\mathbb{R}}

\title{On the minimization of total mean curvature} 
\author{\textsc{J. Dalphin}\footnote{\textit{Institut Elie Cartan de Lorraine UMR CNRS 7502, Universit\'{e} de Lorraine, BP 70239 54506 Vand\oe{}uvre-l\`{e}s-Nancy Cedex, France.} E-mail: \texttt{jeremy.dalphin@mines-nancy.org}, \texttt{antoine.henrot@univ-lorraine.fr}, and \texttt{takeo.takahashi@inria.fr}} , \textsc{A. Henrot}$^{*}$, \textsc{S. Masnou}\footnote{\textit{Institut Camille Jordan UMR CNRS 5208, Universit\'e de Lyon 1, 43 boulevard du 11 novembre 1918, 69622 Villeurbanne Cedex, France.} E-mail: \texttt{masnou@math.univ-lyon1.fr}} , and \textsc{T. Takahashi}$^{*}$.}

\begin{document}
\maketitle

\begin{abstract}
In this paper we are interested in possible extensions of an inequality due to Minkowski:
$\int_{\partial\Omega} H\,dA \geq \sqrt{4\pi A(\partial\Omega)}$ valid for any regular open set
$\Omega\subset\mathbb{R}^3$, where $H$ denotes the scalar mean curvature and $A$ the area.
We prove that this inequality holds true for axisymmetric domains which are convex in the
direction orthogonal to the axis of symmetry. We also show that this inequality cannot be true
in more general situations. However we prove that 
$\int_{\partial\Omega} |H|\,dA \geq \sqrt{4\pi A(\partial\Omega)}$ remains true for any axisymmetric
domain.
\end{abstract}

{\bf keywords :} Total mean curvature, Minkowski inequality, shape optimization, geometric inequality

\section{Introduction}
In 1901, Minkowski proved that the following inequality holds for any non-empty bounded open convex subset $\Omega\subset\mathbb{R}^{3}$ whose boundary $\partial \Omega$ is a $C^{2}$-surface:
\begin{equation}
\label{inegalite_minkowski}
\int_{\partial \Omega} H dA \geqslant \sqrt{4 \pi A(\partial \Omega)},
\end{equation}
where the integration of the scalar mean curvature $H = \frac{1}{2}(\kappa_{1} + \kappa_{2})$ is done with respect to the two-dimensional Hausdorff measure referred to as $A(.)$.
\bigskip

Announced in \cite{MinkowskiFR}, Inequality \eqref{inegalite_minkowski} is proved in \cite[\S 7]{MinkowskiDE} assuming $C^{2}$-regularity. The proof can also be found in \cite[Chapter 6, Exercise (10)]{MontielRos} in the case of ovaloids, i.e. compact simply-connected $C^{\infty}$-surfaces whose Gaussian curvature is positive everywhere.

\medskip
The original proof of Minkowski is based on the isoperimetric inequality together with Steiner-Minkowski formulae. Hence, Inequality \eqref{inegalite_minkowski} remains true if $\partial \Omega$ is only a $C^{1,1}$-surface (or equivalently, if $\partial \Omega$ has a positive reach). If we do not assume any regularity, the same inequality holds with the total mean curvature replaced by mean width.

\medskip
Equality holds in \eqref{inegalite_minkowski} if and only if $\Omega$ is an open ball. This was stated by Minkowski in {\cite[\S 7]{MinkowskiDE}} without proof. A proof due to Favard can be found in \cite[Section 19]{Favard} based on a Bonnesen-type inequality involving mixed volumes. In the appendix, {we give a proof of inequality \eqref{inegalite_minkowski}, with the case of equality, in the axisymmetric situation}, inspired by Bonnesen \cite[Section VI, \S 35 (74)]{Bonnesen}.

\medskip
Inequality \eqref{inegalite_minkowski} is actually a consequence of a generalization due to Minkowski of the isoperimetric inequality. This generalization uses the notion of mixed volumes of convex bodies. We refer to \cite[Theorem 6.2.1, Notes for Section 6.2]{Schneider} and \cite[Sections 49,52,56]{BonnesenFenchel} for a more detailed exposition on that question.

\medskip
In this paper, we are mainly interested in the validity of \eqref{inegalite_minkowski} under other various assumptions, and on the related problem of minimizing the total mean curvature with area constraint: 
\begin{equation}
\label{min_intH}
\inf_{\substack{\Sigma \in \mathfrak{C} \\ A(\Sigma) =  A_{0}}} \int_{\Sigma} H dA, 
\end{equation} 
for a suitable class $\mathfrak{C}$ of surfaces in $\mathbb{R}^{3}$. 
\bigskip

A motivation for Problem \eqref{min_intH} is the study of the Canham-Helfrich energy, used in biology to model the shape of a large class of membranes:
 \[ \mathcal{E}(\Sigma) = \int_{\Sigma} \left( H - H_{0} \right)^{2} dA = \int_{\Sigma} H^{2} dA  - 2 H_{0} \int_{\Sigma} H dA + H_{0}^{2} A_{0}, \]
 where $H_{0}\in\R$ is a fixed constant (called the spontaneous curvature) and $ A_{0} =  A(\Sigma) $ is the area of the membrane. In the particular case of membranes with negative spontaneous curvature $H_0<0$, one can wonder whether the minimization of $\mathcal{E}$ with area constraint {can be done by minimizing individually each term}. Since the Willmore energy $\int_{\Sigma} H^{2} dA$ is invariant with respect to rescaling, and spheres are the only global minimizers, this reduction makes sense only if spheres are also the only solutions to Problem \eqref{min_intH}. We prove in this paper that this is true if the problem is tackled in a {particular} class of surfaces.

\medskip
Let us first introduce two classes of embedded $2$-surfaces in $\R^3$: the class $\mathcal{A}_{1,1}$ of all compact surfaces which are boundaries of axisymmetric sets (i.e. sets with rotational invariance around an axis), and the subclass $\mathcal{A}_{1,1}^+$ of  \textit{axiconvex} surfaces, i.e. surfaces which are boundaries of axisymmetric sets whose intersection with any plane orthogonal to the symmetry axis is either a disk or empty. We first prove the following:

\begin{theorem}
\label{thm_inequality_axiconvex}
Consider the class $\mathcal{A}_{1,1}^{+}$ of axiconvex $C^{1,1}$-surfaces in $\mathbb{R}^{3}$. Then:
\[ \forall \Sigma \in \mathcal{A}_{1,1}^{+}, \quad \int_{\Sigma} H dA \geqslant \sqrt{4 \pi A(\Sigma)},  \]
where the equality holds if and only if $\Sigma$ is a sphere. In particular, for any $A_0>0$:
\[ \int_{\mathbb{S}_{A_{0}}} H dA = \min_{\substack{\Sigma \in \mathcal{A}_{1,1}^{+} \\ A(\Sigma) =  A_{0}}} \int_{\Sigma} H dA = \sqrt{4 \pi A_{0}}, \]
and the sphere $\mathbb{S}_{A_{0}}$ of area $A_{0}$ is the unique global minimizer of this problem.
\end{theorem}

We show then that this result cannot be extended to the general class of compact simply-connected $C^{1,1}$-surfaces in $\mathbb{R}^{3}$, and we even provide a negative clue for the extension to $\mathcal{A}_{1,1}$. More precisely:

\begin{theorem}
\label{thm_no_global_minimizer}
Let $A_0>0$. There exists a sequence of $C^{1,1}$-surfaces $(\Sigma_{n})_{n\in\mathbb{N}}$ and a sequence of axisymmetric $C^{1,1}$-surfaces $(\widetilde{\Sigma}_{n})_{n\in\mathbb{N}}\subset\mathcal{A}_{1,1}$ such that $ A( \Sigma_{n} )= A( \widetilde{\Sigma}_{n} ) = A_{0} $ for any $n\in\mathbb{N}$ with:
\[ \lim_{n \rightarrow +\infty} \int_{\Sigma_{n}} H dA =  - \infty \qquad \mathrm{and} \qquad \lim_{n \rightarrow +\infty} \int_{\widetilde{\Sigma}_{n}} H dA  = 0^{+}. \]
\end{theorem}

\noindent It follows obviously that:
\[ \displaystyle{\inf_{\substack{\Sigma \in C^{1,1} \\ A( \Sigma ) =  A_{0}}} \int_{\Sigma} H dA = - \infty} \qquad \mathrm{and} \qquad \displaystyle{\inf_{\substack{\Sigma \in \mathcal{A}_{1,1} \\ A( \Sigma ) =  A_{0}}} \left| \displaystyle{\int_{\Sigma} HdA} \right|= 0}. \] 
Therefore, Problem \eqref{min_intH} has no solution in the class of (compact simply-connected) $C^{1,1}$-surfaces, and there is good reason to think that it might be the same within the class $\mathcal{A}_{1,1}$, but we were not able to prove it. 
\bigskip

However, although Problem \eqref{min_intH} has no global minimizer, it is easily seen that the sphere $\mathbb{S}_{A_{0}}$ of area $A_{0} $ is a local minimizer of \eqref{min_intH} in the class of {$C^{2}$-surfaces} (Remark \ref{coro_sphere_local_minimizer}) and it can also be proved that $\mathbb{S}_{A_{0}}$ is the unique critical point of \eqref{min_intH} in the class of $C^{3}$-surfaces (Theorem \ref{coro_sphere_point_critique}) by computing the first variation of total mean curvature and of area (Proposition \ref{thm_derivee_intH}).
\bigskip

Hence, this leads us naturally to consider another problem:
\begin{equation}
\label{min_int_absH}
\inf_{\substack{\Sigma \in \mathcal{A}_{1,1} \\ A( \Sigma ) =  A_{0}}} \int_{\Sigma} \vert H \vert  dA,
\end{equation}
for which we can prove:

\begin{theorem}
\label{thm_min_int_absH}
Let $\mathcal{A}_{1,1}$ denotes the class of axisymmetric $C^{1,1}$-surfaces in $\mathbb{R}^{3}$, then:
\[ \forall \Sigma \in \mathcal{A}_{1,1}, \quad \int_{\Sigma} \vert H \vert dA \geqslant \sqrt{4 \pi A( \Sigma )}, \]
where the equality holds if and only if $\Sigma$ is a sphere. In particular, for any  $A_0>0$:
\[ \int_{\mathbb{S}_{A_{0}}} \vert H \vert dA = \min_{\substack{\Sigma \in \mathcal{A}_{1,1} \\ A( \Sigma ) = A_{0}}} \int_{\Sigma} \vert H \vert dA = \sqrt{4 \pi A_{0}},  \]
and the sphere $\mathbb{S}_{A_{0}}$ of area $A_{0}$ is the unique global minimizer of this problem.
\end{theorem}

Let us note that in 1973, Michael and Simon established in \cite{SimonMichael} a Sobolev-type inequality for $m$-dimensional $C^{2}$-submanifolds of $\mathbb{R}^{n}$, for which the case $m=2$ and $n = 3$ with $f \equiv 1$ gives the following inequality:
\[ \int_{\Sigma} \vert H \vert dA \geqslant c_{0} \sqrt{ A( \Sigma )}. \]
More precisely, the constant appearing in the above inequality is $c_{0} = \frac{1}{4^{3}} \sqrt{4 \pi}$ \cite[Theorem 2.1]{SimonMichael}. The better constant $c_{0} = \frac{1}{2} \sqrt{2 \pi}$ was obtained by Topping in \cite[Lemma 2.1]{Topping} and does not seem optimal. From Theorem \ref{thm_min_int_absH}, we think that an optimal constant should be $c_{0} = \sqrt{4 \pi}$. 
\bigskip

We refer to the appendix of \cite{Topping} for a concise proof of the above inequality using Simon's ideas. We also mention \cite[Theorems 3.1, 3.2]{Castillon} for a weighted version of this inequality but less sharp as mentioned in the last paragraph of \cite[Section 3.2]{Castillon}.
\bigskip

The paper is organized as follows. We summarize in Table \ref{table1} several results and open questions related to Problems \eqref{min_intH} and \eqref{min_int_absH} (the term ``inner-convex'' refers to a closed surface which encloses a convex set). In Section \ref{sec_not}, the notation used throughout the text is introduced and the basic definitions of {surface, axisymmetry,} and axiconvexity are recalled. Then, in Sections \ref{sec_thm_axi} and \ref{sec_thm_contrex}, we respectively give the proofs of Theorems \ref{thm_inequality_axiconvex} and \ref{thm_no_global_minimizer}. In Section \ref{sec_sphere}, we study the optimality of the sphere for Problem \eqref{min_intH}. Finally, Theorem \ref{thm_min_int_absH} is proved in Section \ref{sec_thm_absH} and the Minkowski inequality \eqref{inegalite_minkowski} in the axisymmetric case is established in the appendix, where the equality case is also considered. 

\begin{table}
\centering
\begin{tabular}{|l|l|l|}
\hline
\textbf{Class of surfaces $\Sigma$} & \textbf{Assertion}  & \textbf{Proof} \\ 
\hline
$C^{1,1}$ compact inner-convex & $\displaystyle{\int_{\Sigma} H dA \geqslant \sqrt{4 \pi A( \Sigma )}}$ (equality iff $\Sigma$ sphere) & See \cite{MinkowskiDE}, \cite{Favard} \\
\hline
$C^{1,1}$ axisymmetric inner-convex & $\displaystyle{\int_{\Sigma} H dA \geqslant \sqrt{4 \pi A( \Sigma )}}$ (equality iff $\Sigma$ sphere) & See \cite{Bonnesen} \\
\hline
$C^{1,1}$ axiconvex & $\displaystyle{ \int_{\Sigma} HdA \geqslant \sqrt{4 \pi A(\Sigma)}  }$ (equality iff $\Sigma $ sphere) & Theorem \ref{thm_inequality_axiconvex} \\
\hline
\hline
$C^{1,1}$ axisymmetric & $ \displaystyle{ \inf_{A( \Sigma ) = A_{0}} } ~ \begin{array}{|c|} \displaystyle{ \int_{\Sigma} H dA  } \\ \end{array} = 0 $ &  Theorem \ref{thm_no_global_minimizer} \\
\hline
$C^{1,1}$ axisymmetric & $ \displaystyle{ \int_{\Sigma} H dA > 0} $ & \textsc{\textbf{open}} \\
\hline
$C^{1,1}$ compact simply-connected & $ \displaystyle{ \inf_{A( \Sigma) = A_{0}} \int_{\Sigma} HdA = -\infty }$  &  Theorem \ref{thm_no_global_minimizer} \\
\hline
\hline
$C^{2}$ compact simply-connected & $\mathbb{S}_{A_{0}}$ is a local minimizer of $\displaystyle{ \inf_{A (\Sigma ) = A_{0}} \int_{\Sigma} H dA}$ & Remark \ref{coro_sphere_local_minimizer} \\
\hline
$C^{3}$ compact simply-connected & $\mathbb{S}_{A_{0}}$ unique critical point of $\displaystyle{ \inf_{A (\Sigma ) = A_{0}} \int_{\Sigma} H dA}$ & Theorem \ref{coro_sphere_point_critique} \\
\hline
\hline
$C^{1,1}$ axisymmetric  & $  \displaystyle{ \int_{\Sigma} \vert H \vert dA \geqslant \sqrt{4 \pi A( \Sigma)}}$ (equality iff $\Sigma $ sphere) & Theorem \ref{thm_min_int_absH} \\
\hline
$C^{2}$ compact simply-connected & $  \displaystyle{ \int_{\Sigma} \vert H \vert dA \geqslant \sqrt{ \dfrac{\pi}{2}  A(  \Sigma )}   }$ & See \cite{SimonMichael}, \cite{Topping} \\
\hline
$C^{1,1}$ compact simply-connected & $  \displaystyle{ \int_{\Sigma} \vert H \vert dA \geqslant \sqrt{ 4 \pi A( \Sigma )}}$ (equality iff $\Sigma $ sphere) & \textsc{\textbf{open}} \\
\hline
\end{tabular}
\caption{minimizing $\int H$ or $\int |H|$ with area constraint.} 
\label{table1} 
\end{table}

\section{Definitions and notation}
\label{sec_not}
We refer to Montiel and Ros \cite[Definition 2.2]{MontielRos} for the definition of $C^{k,\alpha}$-surfaces without boundary embedded in $\mathbb{R}^3$. We only consider here surfaces homeomorphic to spheres, {i.e. compact and simply-connected}.
\bigskip

In this paper, we present several results on the particular class of $C^{1,1}$ axisymmetric surfaces. We focus on embedded axisymmetric surfaces which are obtained by rotating a planar open simple curve around the segment joining its ends, assuming that the segment meets the curve at no other point.
\bigskip

We choose the $(xz)$-plane as the curve plane and the $z$-line as the rotation axis. We denote by $L > 0$ the total length of the curve. We assume that the following parametrization holds for the curve (using the arc length $s$):
\[ \begin{array}{rrcl}
\gamma: & [0,L] & \longrightarrow & \mathbb{R}^{2} \\
 & s & \longmapsto & \gamma(s) = \left( \begin{array}{c} x(s)\\ z(s)\\ \end{array} \right), \\
\end{array} \]
and we assume without loss of generality that $\gamma(0)=(0,0)$. The axisymmetric surface $\Sigma$ spanned by the rotation of $\gamma$ is the surface $\Sigma$ parametrized by:
\begin{equation} 
\label{parametrisation_X}
\begin{array}{rrcl}
X: & [0,L]\times[0,2\pi[ & \longrightarrow & \mathbb{R}^{3} \\
 & (s,t) & \longmapsto & X(s,t) = \left( \begin{array}{c} x(s)\cos t\\ x(s) \sin t\\ z(s) \\ \end{array} \right), \\
\end{array}
\end{equation}
where $t$ refers to the rotation angle about the $z$-axis. It is well-known that all geometric quantities can be expressed with respect to the angle $\theta$ between the $x$-axis and the tangent line to the curve. This defines a Lipschitz continuous map $\theta : [0,L] \rightarrow \mathbb{R}$ such that:
\[ \forall s \in [0,L], \quad \left( \begin{array}{c} \dot{x}(s) \\ \dot{z}(s) \\ \end{array} \right)  =  \left( \begin{array}{c} \cos \theta(s) \\ \sin \theta(s) \\ \end{array} \right),   \]
therefore, recalling that $x(0)=z(0)=0$,
\begin{equation}
\label{x_z_expression} 
\forall s \in [0,L], \quad  x(s) = \int_{0}^{s} \cos \theta(t) dt \quad \mathrm{and} \quad z(s) = \int_{0}^{s} \sin \theta(t) dt. 
\end{equation}
We also have $dA = 2 \pi x(s) ds$, where $dA$ is the infinitesimal area surface element. Moreover, applying Rademacher's Theorem, the principal curvatures $\kappa_{1}$ and $\kappa_{2}$, implicitly defined by the scalar mean curvature $H = \frac{1}{2}(\kappa_{1}+\kappa_{2})$ and the Gaussian curvature $K = \kappa_{1} \kappa_{2}$, exist almost everywhere and are explicitly given by:
\[ \mbox{for a.e. } s  \in [0,L], \quad  \kappa_{1}(s) = \dfrac{\sin \theta (s)}{x(s)} \quad \mathrm{and} \quad \kappa_{2}(s) = \dot{\theta}(s) \]
{Therefore total mean curvature $\int_\Sigma H dA$ and area $A(\Sigma)$ are given by:
\begin{equation}\label{eq-intH}
\int_\Sigma H dA = \pi \int_0^L \sin\theta(s) + \dot{\theta}(s)x(s)\,ds, \qquad
A(\Sigma)=2\pi \int_0^L x(s)\,ds.
\end{equation}}
All these expressions can be found for example in \cite[Section 3.3, Example 4]{DoCarmo}. Note that the signs of $\kappa_{1}$ and $\kappa_{2}$ depend on the chosen orientation. Throughout the article, the Gauss map always represents the outer unit normal field to the surface. Hence, on the sphere of radius $R > 0$, one can check that $\theta(s) = \frac{s}{R}$ and $\kappa_{1}(s) = \kappa_{2}(s) = \frac{1}{R}$.

\begin{definition}
\label{definition_axisymetric_surface}
We say that $\Sigma$ is an axisymmetric $C^{1,1}$-surface and we write $\Sigma \in \mathcal{A}_{1,1}$ if it is generated as above by a Lipschitz continuous map $\theta: [0,L] \rightarrow \mathbb{R}$, which is admissible in the sense that the following three properties are fulfilled:
\begin{itemize}
\item[(i)] the map $\theta$ satisfies the boundary conditions $\theta(0) = 0$ and $\theta (L) = \pi$; 
\item[(ii)] the map $\gamma$ obtained from $\theta$ satisfies $x(0)=x(L) = 0$ and $z(L) > z(0) = 0$; 
\item[(iii)] the map $\gamma$ is one-to-one on $]0,L[$ and satisfies $x(s) > 0$ for any $s \in ]0,L[$.
\end{itemize}
In particular, $\Sigma$ has no boundary and no self-intersection.
\end{definition}

\begin{definition}
\label{definition_axi_conv_surface}
We say that $\Sigma$ is an axiconvex $C^{1,1}$-surface and we write $\Sigma \in \mathcal{A}_{1,1}^{+}$ if $\Sigma \in \mathcal{A}_{1,1}$ and if the generating map $\theta$ is valued in $[0,\pi]$. In that case the intersection of the surface with any plane orthogonal to
the axis of symmetry is either a circle or a point or the empty set.
\end{definition}

It is easy to check the strict inclusions: (convex and axisymmetric) $\subset$ axiconvex $\subset$ axisymmetric and to prove that an axisymmetric surface is axiconvex if and only if the ordinate function $z$ is non-decreasing, also if and only if it is inner-convex in any direction orthogonal to the axis of revolution.

\section{Proof of Theorem \ref{thm_inequality_axiconvex}}
\label{sec_thm_axi}
First, we note that any axiconvex $C^{1,1}$-surface $\Sigma $ is generated by an admissible Lipschitz continuous map $\theta : [0,L] \rightarrow [0 , \pi]$ as in Section \ref{sec_not} (and $L > 0$ refers to the total length of the generating curve) with the following conditions:
\begin{gather}
\theta(0)=0, \quad \theta(L)=\pi, \label{eq:01} \\
 \int_{0}^{L} \sin \theta(t) dt > 0, \quad   \int_{0}^{L} \cos \theta(t) dt = 0, \label{eq:02}\\
\forall s\in ]0,L[, \quad \int_{0}^{s} \cos \theta(t) dt > 0. \label{eq:03}
\end{gather} 
The first condition of \eqref{eq:02} is verified if \eqref{eq:01} holds and if $\theta([0,L]) \subset [0 , \pi]$. The above conditions are also sufficient to obtain a $C^{1,1}$-axiconvex surface from $\theta : [0,L] \rightarrow [0,\pi]$. Indeed, the fact that the curve obtained from $\theta$ is simple can be deduced from this result.

\begin{proposition}
\label{condition_iii}
Consider $ L > 0$ and a continuous function $u: [0,L] \rightarrow [0, + \infty[$ generating a curve via the $C^{1}$-map $\gamma: s \in [0,L] \mapsto  (\int_{0}^{s} \cos u(\tau) d\tau, \int_{0}^{s} \sin u(\tau)d\tau) $. If $u$ is valued in $[0,\pi]$, then $\gamma$ is a diffeomorphism. In particular, for every distinct $s,t\in ]0,L[$:
\[   \left( \int_{s}^{t} \cos u(\tau) d\tau \right)^{2} + \left( \int_{s}^{t} \sin u(\tau) d\tau \right)^{2} > 0. \]
\end{proposition}

\begin{proof}
The map $\gamma$ can be identified with the differentiable map $s\in[0,L]\mapsto \int_0^s e^{iu(\tau)}d\tau$. Obviously, $|\gamma'(s)|=1$ for every $s\in[0,L]$. If $u$ is valued in $[0,\pi]$, by the mean value theorem for vector-valued functions (see for instance \cite{McLeod}), $\gamma$ is one-to-one, and therefore a diffeomorphism by the global inversion theorem. 
\end{proof}

We also notice that the inner domain of $\Sigma$ associated with $\theta:[0,L] \rightarrow [0, \pi]$ satisfying \eqref{eq:01}, \eqref{eq:02}, and \eqref{eq:03} is a non-empty bounded open subset of $\mathbb{R}^{3}$ which is convex if and only if $\theta$ is non-decreasing. Indeed, in that case, the two principal curvatures are non-negative almost everywhere:
\[\kappa_{1}(s) = \dfrac{\sin \theta(s) }{x(s)} \geqslant 0 \qquad \mathrm{and} \qquad \kappa_{2}(s) = \dot{\theta}(s) \geqslant 0 \qquad \mathrm{a}.\mathrm{e}. \]

We prove Theorem \ref{thm_inequality_axiconvex} by using a non-decreasing rearrangement of $\theta$:
\begin{equation}
\label{eq:04}
\forall s \in [0,L], \quad \theta^{*}(s) = \sup \left\lbrace c \in [0,\pi], \quad s \in \left[ L - \vert \left\lbrace t \in [0,L], \quad \theta(t) \geqslant c \right\rbrace \vert, L \right] \right\rbrace,     
\end{equation}
where $\vert  ~.~  \vert $ refers here to the one-dimensional Lebesgue measure. We split the proof into the following three steps:
\begin{enumerate}
\item We check that $\theta^{*}$ generates an axisymmetric inner-convex $C^{1,1}$-surface $\Sigma^{*}$.
\item We show that: 
\[ \int_{\Sigma} H dA = \int_{\Sigma^{*}} H dA \geqslant \sqrt{ 4 \pi A( \Sigma^{*} )} \geqslant \sqrt{4 \pi A( \Sigma )}. \]
\item We study the equality case.
\end{enumerate}
It is convenient to first recall some well-known results about rearrangements.

\begin{proposition}
\label{Parrangements}
Consider any Lipschitz continuous map $u: [0,L] \rightarrow [0 , \infty[$ and its non-decreasing rearrangement $u^{*}$ defined by: 
\[ \forall s \in [0,L], \quad u^{*}(s) = \sup \left\lbrace c \in [0,\infty[, \quad s \in \left[ L - \vert \left\lbrace t \in [0,L], \quad u(t) \geqslant c \right\rbrace \vert, L \right] \right\rbrace.  \]
Then, the following properties hold true. 
\begin{enumerate}
\item The map $u^{*}$ is non-decreasing.
\item The map $u^{*}$ is Lipschitz continuous with the same Lipschitz modulus as $u$.
\item For any continuous map $F: [0, + \infty[ \rightarrow \mathbb{R}$, we have the following equality: 
\[ \int_{0}^{L} F(u(s))ds = \int_{0}^{L} F(u^{*}(s))ds. \]
\item For any continuous increasing map $F: [0, + \infty[ \rightarrow [0,+\infty[$, we have $(F(u))^{*} = F(u^{*})$.
\item (Hardy--Littlewood inequality) If $v : [0,L] \rightarrow [0,+\infty[$ is another Lipschitz continuous map and $v^{*}$ denotes its non-decreasing rearrangement, then:
\[ \int_{0}^{L} u(s) v(s) ds \leqslant \int_{0}^{L} u^{*}(s) v^{*}(s) ds.\]
\end{enumerate}
\end{proposition}
\begin{proof}
The above results are quite classical. We refer to \cite{Kawohl,Kesavan} for proofs and references. The first property corresponds to \cite[Proposition 1.1.1]{Kesavan}. The second one is proved in \cite[Lemma 2.3]{Kawohl}. The third and fourth one are respectively established in \cite{Kawohl} II.2 Property (C) and \cite[Proposition 1.1.4]{Kesavan}. Concerning the Hardy-Littlewood inequality, a proof can be found in \cite[Theorem 1.2.2]{Kesavan} or in \cite{Kawohl} II.2 Property (P1). To be a bit more precise, the proofs are generally written for a non-increasing rearrangement of $u$. For instance, the rearrangement of $u$ is defined in \cite{Kesavan} by:
\[\forall s \in [0,L], \quad u^{\sharp}(s):=\inf \left\lbrace c \in [0,\infty[, \quad s > \vert \left\lbrace t \in [0,L], \quad u(t) > c \right\rbrace \vert \right\rbrace. \]
However, one can notice that: 
\[ \forall s \in [0,L],\quad u^{*}(s) = u^{\sharp}(L-s), \]
and adapt the proofs of \cite{Kesavan} in order to deduce the above properties.
\end{proof}

\begin{proof}[Proof of Theorem \ref{thm_inequality_axiconvex}]
\noindent \textbf{Step 1:} the map $\theta^{*}$ defined by \eqref{eq:04} generates an axisymmetric inner-convex $C^{1,1}$-surface $\Sigma^{*}$.
\bigskip

We only need to check \eqref{eq:01}, \eqref{eq:02}, and \eqref{eq:03} for $\theta^{*}$. Assertion \eqref{eq:01} follows from the definition of $\theta^{*}$ given in \eqref{eq:04}. We define the functions: 
\[ \forall s \in [0,L], \quad x_{*}(s) = \int_{0}^{s} \cos \theta^{*}(t)  dt \qquad \mathrm{and} \qquad z_{*}(s) = \int_{0}^{s} \sin \theta^{*}(t)  dt. \] 
Note that $x_{*}, z_*$ {\it are not} the rearrangements of $x,z$. From Property 3 in Proposition \ref{Parrangements}, we get $x_*(L)=x(L)=0$ and $z_*(L)=z(L)>0$ then the relations in \eqref{eq:02} hold true for $\theta^*$.  Relation \eqref{eq:03} is equivalent to $x_{*}(s) > 0$ for any $s \in ]0,L[$. {Since $\dot{x}_{*} = \cos \theta^{*}$}, Property 1 in Proposition \ref{Parrangements} combined with the fact that $\theta^{*}([0,L]) \subseteq [0, \pi]$ ensures $x_{*}$ is a concave map, not identically zero. Hence, $x_{*}>0$ in $]0,L[$.
\bigskip

\noindent \textbf{Step 2:} we compare the total mean curvature and the area of $\Sigma$ with the ones of $\Sigma^{*}$.
\bigskip

First, observe that we can obtain from an integration by parts:
\[  \int_{\Sigma} H dA  =  \int_{0}^{L} \dfrac{1}{2} \left( \dfrac{\sin \theta (s)}{x(s)} + \dot{\theta}(s) \right) 2 \pi x(s) ds  = \pi \int_{0}^{L} F (\theta(s)) ds, \]
where $F$ is the continuous map $ x \mapsto \sin x - x \cos x$. Using Property 3 in Proposition \ref{Parrangements}, we deduce that:
\begin{equation}  
\label{eq_F}
\int_{\Sigma} H dA = \int_{\Sigma^{*}} H dA. 
\end{equation}
Now, since $\Sigma^{*}$ is an {axisymmetric} inner-convex $C^{1,1}$-surface, we can apply the Minkowski Theorem, see \eqref{inegalite_minkowski} {or Corollary \ref{coro_minkowski_axi_bonnesen}:}
\begin{equation} 
\label{equ_min}
\int_{\Sigma^{*}} H d A \geqslant \sqrt{ 4 \pi A( \Sigma^{*} )}.
\end{equation} 
Then, we need to compare the areas of $\Sigma$ and $\Sigma^{*}$. For that purpose, we are going to use the Hardy-Littlewood inequality combined with the following observation coming from an integration by parts:
\[ A ( \Sigma ) = \int_{\Sigma} dA = \int_{0}^{L} 2 \pi x(s) ds = - 2 \pi \int_{0}^{L} s \cos \theta(s) ds. \]
Set $u(s) = s$ and $v(s) = 1 - \cos \theta(s)$ for every $s \in [0,L]$. These two functions being non-negative and Lipschitz continuous, we get from Property 5 of Proposition \ref{Parrangements}:
\[ \int_{0}^{L} u(s) v(s) ds \leqslant \int_{0}^{L} u^{*}(s) v^{*}(s) ds, \]
where $u^{*}$ and $v^{*}$ are the non-decreasing rearrangements of $u$ and $v$, respectively. Since the continuous map $x \mapsto 1 - \cos x$ is non-negative and increasing on $[0, \pi]$, we use Property 4 in Proposition \ref{Parrangements} in order to get $v^{*} =(1- \cos (\theta))^{*} = 1 - \cos (\theta^{*}) $ but we have also $u^{*}(s) = u(s) = s$. Finally, we obtain that:
\begin{equation}  
\label{equ_area}
\dfrac{L^{2}}{2} + \dfrac{A(\Sigma)}{2 \pi } = \int_{0}^{L} s ( 1 - \cos \theta(s) ) ds \leqslant \int_{0}^{L} s ( 1 - \cos  \theta^{*}(s) ) ds = \dfrac{L^{2}}{2} + \dfrac{A( \Sigma^{*} )}{2 \pi}. 
\end{equation}
Combining \eqref{eq_F}, \eqref{equ_min}, and \eqref{equ_area}, the inequality of Theorem \ref{thm_inequality_axiconvex} is therefore established:
\[ \int_{\Sigma} H dA = \int_{\Sigma^{*}} H dA \geqslant \sqrt{ 4 \pi A( \Sigma^{*} )} \geqslant \sqrt{4 \pi A( \Sigma )}. \]
\bigskip

\noindent \textbf{Step 3:} the equality case.
\bigskip

Assume that there exists $\Sigma \in \mathcal{A}^{+}_{1,1}$ such that the equality holds in the previous inequalities. Then, we have:
\begin{equation}
\label{eq_equality} 
\int_{\Sigma} H dA  = \int_{\Sigma^{*}} H dA = \sqrt{4 \pi A( \Sigma^{*} )} = \sqrt{4 \pi A( \Sigma )}. 
\end{equation}
Therefore, since $\Sigma^{*}$ is an inner-convex $C^{1,1}$-surface, using the Minkowksi Theorem, we deduce that $\Sigma^{*}$ must be a sphere (equality in \eqref{inegalite_minkowski}, see Corollary \ref{coro_minkowski_axi_bonnesen}). Now, we show that $\Sigma \equiv \Sigma^{*}$ i.e. $\theta = \theta^{*}$. From \eqref{equ_area} and \eqref{eq_equality}, we have the equality: 
\[ \int_{0}^{L} s v(s)ds = \int_{0}^{L} s v^{*}(s)ds,  \]
where the map $v: s \mapsto v(s) = 1 - \cos \theta (s)$ has already been introduced. The above equality and an integration by parts yield to the following relation: 
\begin{equation} 
\label{eq_layer_cake}
\int_{0}^{L}  \left( \int_{s}^{L} v(c) dc \right) ds = \int_{0}^{L}  \left( \int_{s}^{L} v^{*}(c) dc \right) ds.
\end{equation}
Since $\mathbf{1}_{[s,L]}^{*} = \mathbf{1}_{[s,L]}$, the Hardy-Littlewood inequality implies that:
\[ \forall s \in [0,L], \quad \int_{s}^{L} v(c)dc = \int_{0}^{L} \mathbf{1}_{[s,L]}(c) v(c) dc \leqslant \int_{0}^{L} \mathbf{1}^{*}_{[s,L]}(c) v^{*}(c) dc = \int_{s}^{L} v^{*}(c)dc.  \]
Combining the above inequality and \eqref{eq_layer_cake}, we deduce that:
\[ \forall s \in [0,L], \quad \int_{s}^{L} v(c) dc = \int_{s}^{L} v^{*}(c) dc, \]
thus $(1- \cos [\theta^{*}]) = 1 - \cos [\theta]$ and $\theta = \theta^{*}$ on $[0,L]$. Hence, $\Sigma \equiv \Sigma^{*}$ and $\Sigma$ must be a sphere. Conversely, any sphere $\Sigma$ satisfies the equality $\int_{\Sigma} H dA = \sqrt{4 \pi A( \Sigma )}$, which concludes the proof of Theorem \ref{thm_inequality_axiconvex}.
\end{proof}

\section{Proof of Theorem \ref{thm_no_global_minimizer}}
\label{sec_thm_contrex}
In this section, we build two sequences of surfaces of constant area. The first one is not axisymmetric and its total mean curvature tends to $- \infty$ while the other one is axisymmetric and its total mean curvature tends to zero. Figures \ref{boule_trouee} et \ref{double_sphere} describe their respective constructions.

\subsection{Total mean curvature is not bounded from below}
We first compute the total mean curvature of a sphere of radius $R > 0$ where a neighbourhood of the north pole has been removed, and replaced by an internal sphere of small radius $\varepsilon > 0$. The two parts are glued so that the resulting surface referred to as $\Sigma_{\varepsilon}$ is an axisymmetric $C^{1,1}$-surface illustrated in Figure \ref{boule_trouee}. 

\begin{figure}[h!]
\centering
\includegraphics{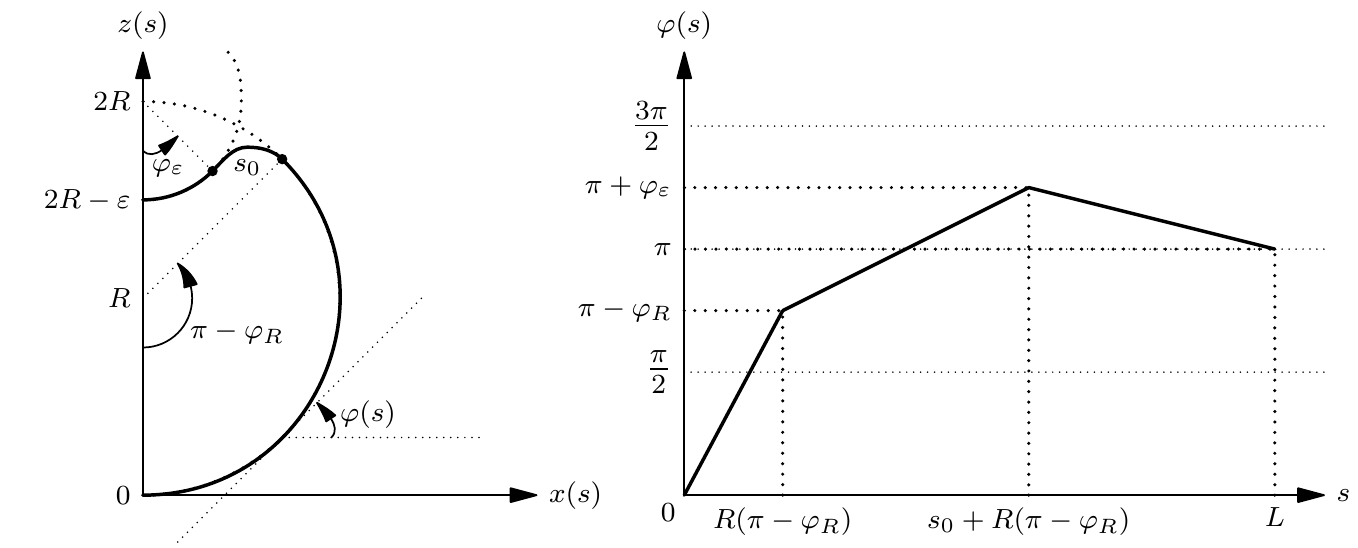}
\caption{the construction of the sequence of axisymmetric surfaces $(\Sigma_{\varepsilon})_{\varepsilon > 0}$.}
\label{boule_trouee}
\end{figure}

More precisely, let us fix $\varphi_{\varepsilon}=\frac{\pi}{2}\,-\varepsilon$ and let us consider the function $\varphi: [0,L] \rightarrow \mathbb{R}$ defined by:
\[ \varphi(s) = \left\lbrace \begin{array}{ll}
 \dfrac{s}{R}  & \mathrm{if}~ s \in [0,R(\pi-\varphi_{R})] \\
 & \\
  \dfrac{\varphi_{R} + \varphi_{\varepsilon}}{s_{0}} \left( s - R(\pi - \varphi_{R}) \right) + \pi - \varphi_{R}  & \mathrm{if}~ s \in [R(\pi - \varphi_{R}), s_{0} + R(\pi - \varphi_{R})] \\
& \\
 -\dfrac{1}{\varepsilon}(s-s_{0}-R(\pi - \varphi_{R}))+\pi + \varphi_{\varepsilon} & \mathrm{if} ~ s \in [s_{0} + R(\pi - \varphi_{R}),L], \\  
\end{array} \right. \]
with 
\[ \varphi_R, \ \varphi_{\varepsilon}=\frac{\pi}{2}\,-\varepsilon \in \left]0,\frac{\pi}{2}\right[, \quad s_0>0 \quad \mathrm{and} \quad L = \varepsilon \varphi_{\varepsilon} + s_{0} + R(\pi - \varphi_{R}). \]
In the above expression, there are three parameters $\varphi_{\varepsilon}$, $\varphi_R$ and $s_0$, but actually we will have to impose
two extra conditions (\ref{13bis}) and (\ref{14bis}) to express that $x(L)=0$ and $z(L)=2R-\varepsilon$.
The map $\varphi$ is continuous and piecewise linear, and satisfies \eqref{eq:01}, \eqref{eq:02}, \eqref{eq:03}. The surface $\Sigma_{\varepsilon}$ is obtained through formulas \eqref{parametrisation_X}, \eqref{x_z_expression} when $\theta$ is replaced by $\varphi$. The first part of the definition of $\varphi$ generates almost a sphere of radius $R > 0$ since $\varphi_{R}$ will be chosen small. The third part generates almost an internal half-sphere of radius $\varepsilon > 0$. The second part corresponds to the gluing of the two spheres and has a length $s_{0} > 0$. Let us note that $L > 0$ is the total length of the curve.
\bigskip

We compute $x(s) = \int_{0}^{s} \cos\varphi(t)dt $ and $z(s) = \int_{0}^{s} \sin \varphi(t)dt$ and taking into account that
the expression for the last interval describes a part of the sphere of radius $\varepsilon$, we get:
\[ x(s) = \left\lbrace \begin{array}{ll}
R \sin  \varphi(s) & \mathrm{if}~ s \in [0,R(\pi-\varphi_{R})] \\
 & \\
\left( R - \dfrac{s_{0}}{\varphi_{R} + \varphi_{\varepsilon}} \right) \sin \varphi_{R} + \dfrac{s_{0}}{\varphi_{R} + \varphi_{\varepsilon}} \sin \varphi(s) ~~~~~~ & \mathrm{if}~ s \in [R(\pi - \varphi_{R}), s_{0} + R(\pi - \varphi_{R})] \\
& \\
- \varepsilon \sin \varphi(s)  & \mathrm{if} ~ s \in [s_{0} + R(\pi - \varphi_{R}),L], \\  
\end{array} \right. \]
and also
\[ z(s) = \left\lbrace \begin{array}{ll}
R \left( 1 - \cos   \varphi(s) \right) & \mathrm{if}~ s \in [0,R(\pi-\varphi_{R})] \\
 & \\
R + \left( R - \dfrac{s_{0}}{\varphi_{R} + \varphi_{\varepsilon}} \right) \cos \varphi_{R} - \dfrac{s_{0}}{\varphi_{R} + \varphi_{\varepsilon}} \cos \varphi(s)  & \mathrm{if}~ s \in [R(\pi - \varphi_{R}), s_{0} + R(\pi - \varphi_{R})] \\
& \\
2R + \varepsilon \cos  \varphi(s)  & \mathrm{if} ~ s \in [s_{0} + R(\pi - \varphi_{R}),L]. \\  
\end{array} \right. \]
\bigskip

We express now continuity of $x(s)$ and $z(s)$ at $s=s_{0} + R(\pi - \varphi_{R})$. The first relation gives $s_{0}$ explicitly in terms of $\varphi_{R}$ and $\varphi_{\varepsilon}$. The second one gives an implicit relation between $\varphi_{R}$ and $\varphi_{\varepsilon}$. 
\begin{equation}\label{13bis}
\left( R - \dfrac{s_{0}}{\varphi_{R} + \varphi_{\varepsilon}} \right) \sin \varphi_{R} - \dfrac{s_{0}}{\varphi_{R} + \varphi_{\varepsilon}} \sin \varphi_{\varepsilon} = \varepsilon \sin \varphi_{\varepsilon} \quad \text{i.e.}\quad  s_{0} = (\varphi_{R} + \varphi_{\varepsilon}) \dfrac{R \sin \varphi_{R} - \varepsilon \sin \varphi_{\varepsilon}}{\sin \varphi_{R} + \sin \varphi_{\varepsilon} }, 
\end{equation}
and 
\begin{equation}\label{14bis}
R + \left( R - \dfrac{s_{0}}{\varphi_{R} + \varphi_{\varepsilon}} \right) \cos \varphi_{R} + \dfrac{s_{0}}{\varphi_{R} + \varphi_{\varepsilon}} \cos \varphi_{\varepsilon} = 2 R - \varepsilon \cos \varphi_{\varepsilon} .
\end{equation}
The last relation can be rewritten, using the first relation, in the following form: 
\[ \dfrac{(R+\varepsilon)\cos \varphi_{R} - R}{\sin \varphi_{R}} + \dfrac{(R+\varepsilon)\cos \varphi_{\varepsilon} - R}{\sin \varphi_{\varepsilon}} = 0. \]
To see that this relation can be satisfied, we introduce the map $f: x\in ]0,\frac{\pi}{2}[ \mapsto \frac{(R+\varepsilon) \cos x - R}{\sin x} $, which is smooth, decreasing and surjective. Hence, it is an homeomorphism on its image and the previous relation become with this notation:
\[ f(\varphi_{R}) + f(\varphi_{\varepsilon}) = 0 \Longleftrightarrow \varphi_{R} = f^{-1}(-f(\varphi_{\varepsilon})) .\] 
We recall that $\varphi_{\varepsilon}=\frac{\pi}{2}-\varepsilon$ and we get by a straightforward computation:
\[ f(\varphi_R)=R-R\epsilon+o(\epsilon). \]
Using the expression of $f$, we deduce that $\sin(\varphi_R)= \frac{\varepsilon}{R} + o(\varepsilon)$ and therefore, we obtain:
\begin{equation}
\label{devphiR}
 \varphi_R=\frac{\varepsilon}{R} + o(\varepsilon).
\end{equation}
Now, we can compute the total mean curvature and the area of the surface. We obtain:
\begin{equation}
\label{exp2}
\left\lbrace \begin{array}{l}
\displaystyle{\dfrac{1}{\pi} \int_{\Sigma_{\varepsilon}} H dA = \int_{0}^{L} \left( \sin\varphi(s)+\dot{\varphi}(s)x(s) \right) ds = 4R - \left( 2- \dfrac{\pi}{2} \right) \varepsilon + o(\varepsilon) } \\
\\
 \displaystyle{ \dfrac{A(\Sigma_{\varepsilon})}{2 \pi} = \int_{0}^{L} x(s) ds = 2 R^{2} + \dfrac{\varepsilon^{2}}{2} + o(\varepsilon^{2}).}  \\
\end{array} \right.
\end{equation}
We can notice in the above expressions a first term which is the contribution of the sphere of radius $R$ and a second one due to the half-sphere of radius $\varepsilon$ and the gluing. Note that the gluing has some first order impact on these relations, which is not obvious at first sight. We are now in position to prove the first part of Theorem \ref{thm_no_global_minimizer}.

\begin{proof}[Proof of Theorem \ref{thm_no_global_minimizer}]
We decide to perform many perturbations of that kind all around the sphere. Notice that, for $\varepsilon$ small enough, the perturbation we defined is contained in a ball of radius $\frac{3}{2}\varepsilon$ centred at the north pole. Thus it suffices to count how many such disjoint small balls we can put on the surface of the sphere of radius $R$. We will also use the fact that each perturbation makes a contribution for the total mean curvature and the area as $- \pi(2- \frac{\pi}{2}) \varepsilon$ and $\pi \varepsilon^2$ (respectively) at first order, according to (\ref{exp2}). We will denote by $N_\varepsilon$ the number of perturbations. We first divide the surface of the sphere in slices $S_k$ of latitude between $\frac{2\varepsilon}{R}(2k-1)$ and $\frac{2\varepsilon}{R}(2k+1)$, $k\in \{-K_\varepsilon \ldots,K_\varepsilon\}$ with $K_\varepsilon$ the integer part of {$\frac{\pi R}{8\varepsilon} -\frac{1}{2}$.} The (geodesic) width of each slice is $4\varepsilon$. Now the slice $S_k$ has a mean radius which is {$R\cos(\frac{4 k \varepsilon}{R})$}, thus a perimeter which is {$2\pi R\cos(\frac{4 k \varepsilon}{R})$} and therefore, we can put on it {$[2\pi R\cos(\frac{4k\varepsilon}{R})/4\varepsilon]$ patches of diameter close to $4\varepsilon$, where $[.]$ refers to the integer part}. On each patch, we can center a ball of radius {$\frac{3\varepsilon}{2}$}. Consequently, the total number of patches where we can put disjoint ball of diameter {$3\varepsilon$ is given by:
\begin{equation}
\label{number}
N_{\varepsilon}=\sum_{k=-K_\varepsilon}^{K_\varepsilon -1} \left[ \dfrac{\pi R}{2\varepsilon} \cos \left( \dfrac{4 k \varepsilon}{R} \right) \right].
\end{equation}
Using that $K_\varepsilon$ satisfies 
$$\frac{\pi R}{8\varepsilon} -\frac{3}{2} \leqslant K_\varepsilon \leqslant \frac{\pi R}{8\varepsilon}-\frac{1}{2},$$
we deduce from \eqref{number} that
\begin{equation}\label{taktak01}
N_{\varepsilon}=\frac{\pi R^2}{4\varepsilon^2} + O\left(\frac 1\varepsilon \right).
\end{equation}

Then, the resulting $C^{1,1}$-surface obtained this way (written again $\Sigma_{\varepsilon}$) is compact simply-connected (and not axisymmetric). Moreover, we deduce from \eqref{taktak01}:}
\[ \left\lbrace \begin{array}{l} 
 \displaystyle \int_{\Sigma_{\varepsilon}} H dA  =  4 \pi R - \pi \left( 2 - \dfrac{\pi}{2} \right) N_\varepsilon \varepsilon+ o(N_\varepsilon \varepsilon) 
 = -  \left( 2 - \dfrac{\pi}{2} \right)\frac{\pi^2 R^2}{4\varepsilon} + o\left(\frac 1\varepsilon\right),\\
\displaystyle{ A(\Sigma_{\varepsilon})  = 4 \pi R^{2} + \pi N_\varepsilon\varepsilon^{2} + o(N_\varepsilon\varepsilon^{2}) = 4 \pi R^{2} + \frac{\pi^2 R^2}{4}\, +o(1) }. \\ 
\end{array} \right. \]
Finally, we make a rescaling of $\Sigma_{\varepsilon}$ such that its area is exactly the required area $A_{0}$. First, we set $R > 0$ such that $4 \pi R^{2} = A_{0}$, i.e. the sphere of radius $R$ has area $A_{0}$. Then we set:
\[ t_{\varepsilon} = \sqrt{\dfrac{A_{0}}{A(\Sigma_{\varepsilon})}} = \left(1+\frac{\pi}{16} + o(1)  \right)^{-1/2}. \]
Hence, the surface $t_{\varepsilon} \Sigma_{\varepsilon}$ has area $A_{0}$ and we have:
\[ \int_{t_{\varepsilon} \Sigma_{\varepsilon}} H dA = t_{\varepsilon} \left( \int_{\Sigma_{\varepsilon}} H dA \right)  
= - \left(1+\frac{\pi}{16} \right)^{-1/2}   \left( 2 - \dfrac{\pi}{2} \right)\frac{\pi^2 R^2}{4\varepsilon} + o\left(\frac 1\varepsilon\right).  \]
By letting $\varepsilon$ tend to zero, we thus obtain the first part of Theorem \ref{thm_no_global_minimizer}. The total mean curvature, even constrained by area, is not bounded from below. 
\end{proof}

\subsection{A sequence converging to a double sphere}
We now detail the construction of a sequence $(\widetilde{\Sigma}_{\varepsilon})_{\varepsilon > 0}$ of axisymmetric $C^{1,1}$-surfaces of constant area whose total mean curvature tends to zero, which will end the proof of Theorem \ref{thm_no_global_minimizer}.
\bigskip
\begin{figure}[ht!]
\centering
\includegraphics{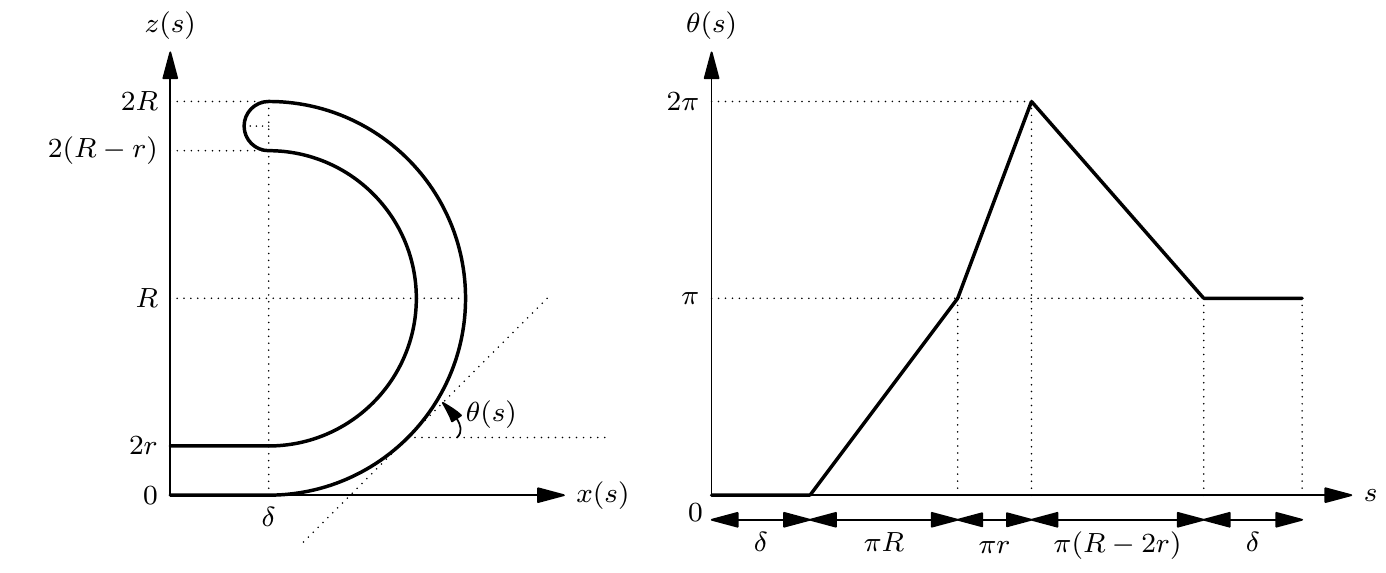}\caption{the construction of the sequence of axisymmetric surfaces $(\widetilde{\Sigma}_{\varepsilon})_{\varepsilon > 0}$.}
\label{double_sphere}
\end{figure}

We consider the sequence of surfaces $(\widetilde{\Sigma}_{\varepsilon})_{\varepsilon > 0}$ described in Figure \ref{double_sphere}. They consist in two spheres of radius $R > 0$ and $R-2r>0$ glued together at a distance $\delta > r > 0$ of the axis of revolution and such that the generating map $\theta:[0,L] \rightarrow \mathbb{R}$ is piecewise linear. More precisely, we have:
\[ \theta(s) = \left\lbrace \begin{array}{ll}
0 & \mathrm{if}~ s \in [0,\delta] \\
& \\ 
 \dfrac{1}{R}(s-\delta)  & \mathrm{if}~ s \in [\delta, \delta + \pi R ] \\
 & \\
  \dfrac{1}{r}(s-\delta-\pi R) + \pi  & \mathrm{if}~ s \in [\delta + \pi R, \delta + \pi(R+r)] \\
& \\
 -\dfrac{1}{R-2r}(s-\delta-\pi R-\pi r) + 2 \pi & \mathrm{if} ~ s \in [ \delta + \pi(R+r), \delta + \pi(2R-r)] \\  
  & \\
  \pi & \mathrm{if} ~ s \in [ \delta + \pi(2R-r),L], \\ 
\end{array} \right. \]
where $L = 2 \delta +  \pi(2R-r) > 0$ is the total length of the generating curve. Then, a computation of $x(s) = \int_{0}^{L} \cos \theta(t) dt $ and $z(s) = \int_{0}^{s} \sin  \theta(t) dt$ gives the following relations:
\[ x(s) =  \left\lbrace \begin{array}{ll}
 s & \mathrm{if}~ s \in [0,\delta] \\
 \delta + R \sin  \theta(s)  & \mathrm{if}~ s \in [ \delta, \delta + \pi R] \\
 \delta + r \sin  \theta(s) & \mathrm{if}~ s \in [ \delta + \pi R, \delta +\pi (R+r)] \\
 \delta - (R-2r) \sin   \theta(s) & \mathrm{if}~ s \in [ \delta + \pi (R+ r), \delta +\pi (2R- r)] \\
 L-s & \mathrm{if} ~ s \in [\delta +\pi (2R- r),L], \\  
\end{array} \right. \]
and also
\[ z(s) = \left\lbrace \begin{array}{ll}
  0 & \mathrm{if}~ s \in [0,\delta] \\
  R \left( 1 - \cos   \theta(s)  \right) & \mathrm{if}~ s \in [ \delta, \delta + \pi R] \\
  2R - r (1+ \cos \theta(s))  & \mathrm{if}~ s \in [ \delta + \pi R, \delta + \pi (R+r)] \\
 2(R -r) - (R-2r)(1 - \cos\theta(s))  & \mathrm{if}~ s \in [ \delta + \pi (R+ r),  \delta + \pi (2R- r)] \\
  2r & \mathrm{if} ~ s \in [ \delta + \pi (2R- r),L]. \\
\end{array} \right. \]
Finally, we obtain the following expressions:
\[ \left\lbrace \begin{array}{l}
\displaystyle{ \int_{\widetilde{\Sigma}_{\varepsilon}} H dA = \pi \int_{0}^{L} \left( \sin \theta(s) +\dot{\theta}(s)x(s) \right)ds = 4 \pi r + \pi^{2} \delta } \\
\\
\displaystyle{ A(\widetilde{\Sigma}_{\varepsilon}) = 2\pi \int_{0}^{L} x(s) ds = 2\pi \delta^{2} + 2\pi^2 \delta(2R-r) + 4\pi \left( R^{2} -  r^{2} +  (R-2r)^{2} \right).  } \\
\end{array} \right. \]
Now, we impose that $\delta = 2r > r > 0$. The last relation is thus a second order polynomial in $R > 0$ and for each (small) $r$, there exists a unique positive root $R_r$ such that $A(\widetilde{\Sigma}_{\varepsilon}) = A_{0}$. Moreover, $R_r$ converges to $R_0=\sqrt{A_0/8\pi}$ when $r \rightarrow 0$. Then, we see that the total mean curvature converges to zero from above as $r $ tends to $0^{+}$, which concludes the proof of Theorem \ref{thm_no_global_minimizer}. 
\section{The sphere is the unique smooth critical point}
\label{sec_sphere}
According to Theorem \ref{thm_no_global_minimizer}, the sphere is not a global minimizer of \eqref{min_intH} in the class of $C^{1,1}$-surfaces. However, in this section, we establish that the sphere is always a smooth local minimizer. Then, we compute the first variation of total mean curvature and area to obtain the Euler-Lagrange equation associated to \eqref{min_intH}. We deduce that the sphere is the unique smooth critical point of \eqref{min_intH}.
\begin{remark}
\label{coro_sphere_local_minimizer}
Since the ball of radius $R$ is a strictly convex set whose boundary has principal curvatures  everywhere equal to $1/R$, any perturbation of class $C^{2}$ of the sphere
yields a perturbation of class $C^0$ of its curvatures and then the perturbed domain remains convex. From \eqref{inegalite_minkowski}, the sphere is a global minimizer of \eqref{min_intH} among compact inner-convex $C^{2}$-surfaces so the sphere is obviously a local minimizer of total mean curvature for small perturbations of class $C^2$. 
\end{remark}

\begin{proposition}[First variation of total mean curvature and area]
\label{thm_derivee_intH}
Assume that $\Sigma$ is a compact simply-connected $C^{2}$-surface. Consider a smooth vector field $ \mathbf{V} : \mathbb{R}^{3} \rightarrow \mathbb{R}^{3}$ and the family of maps $\phi_{t} : \mathbf{x} \in \Sigma \mapsto \mathbf{x} + t \mathbf{V}(\mathbf{x})$. Then, we have:
\[ \dfrac{d}{dt} \left( \int_{\phi_{t}(\Sigma)} 1 dA \right)_{t = 0} = \int_{\Sigma} 2H \left( \mathbf{V} \cdot \mathbf{N} \right) dA, \]
where $\mathbf{N} : \Sigma \rightarrow \mathbb{S}^{2}$ refers to the Gauss map representing the outer unit normal field of $\Sigma$. Moreover, if $\Sigma$ is a compact simply-connected $C^{3}$-surface, then we also get:
\[ \dfrac{d}{dt} \left( \int_{\phi_{t}(\Sigma)} H dA \right)_{t = 0} = \int_{\Sigma} K \left( \mathbf{V} \cdot \mathbf{N} \right) dA, \]
where $K=\kappa_{1} \kappa_{2}$ refers to the Gaussian curvature.
\end{proposition}

\begin{proof}
The first variation of area is classical, see for example \cite[Corollary 5.4.16]{HenrotPierre}. Concerning the first variation of total mean curvature, we refer to \cite[Theorem 2.1]{DoganNochetto} or \cite[Theorem 5.4.17]{HenrotPierre}. Using the notation of \cite{DoganNochetto} i.e. $J(\Sigma) = \int_{\Sigma} H dA$, we get in the case where $\psi(x,\Sigma) $ represents any extension of the scalar mean curvature $H$, and $\psi'(\Omega;\mathbf{V})$ its shape derivative in the direction $\mathbf{V}$: 
\[ dJ(\Sigma;\mathbf{V}) = \int_{\Sigma} \psi'(\Omega;\mathbf{V})\vert_{\Sigma} dA + \int_{\Sigma} (\partial_{\nu}\psi + 2H \psi)VdA.  \]
Now, Lemma 3.1 in \cite{DoganNochetto} states $\psi'(\Sigma;\mathbf{V}) = - \frac{1}{2} \Delta_{\Sigma} V$, where $V=\mathbf{V} \cdot \mathbf{N}$ and $\Delta_{\Sigma} = \mathrm{div}_{\Sigma} \nabla_{\Sigma}$ is the usual Laplace-Beltrami operator. Moreover, from \cite[Lemma 3.2]{DoganNochetto}, and since $\Sigma$ is $C^3$, we get $\partial_{\nu} H = - \frac{1}{2}(\kappa_{1}^{2} + \kappa_{2}^{2}) = -2 H^{2} + \kappa_{1} \kappa_{2}$. Therefore we deduce:
\[ dJ(\Sigma;\mathbf{V}) = - \frac{1}{2}\int_{\Sigma} \Delta_{\Sigma} V dA + \int_{\Sigma}(-2 H^{2} +  \kappa_{1} \kappa_{2} + 2 H^{2})VdA = \int_{\Sigma} \kappa_{1} \kappa_{2} V dA,  \] 
which gives the announced result and concludes the proof of Proposition \ref{thm_derivee_intH}.  
\end{proof}

\begin{theorem}
\label{coro_sphere_point_critique}
Within the class of compact simply-connected $C^{3}$-surfaces, if the area is constrained to be equal to a fixed positive number, then the corresponding sphere is the unique critical point of the total mean curvature.
\end{theorem}

\begin{proof}
Consider any critical point $\Sigma$ of \eqref{min_intH} which is a compact simply-connected $C^{3}$-surface. From Proposition \ref{thm_derivee_intH}, there exists a Lagrange multiplier $\lambda \in \mathbb{R}$ such that $K = 2 \lambda H$. Let us observe that $\lambda \neq 0$ otherwise $K = 0$ which is not possible (indeed, any compact surface has a point where $K > 0$ \cite[Exercise 3.42]{MontielRos}). Now assume that $\lambda < 0$. Then, from the relation $H^{2} = (\frac{\kappa_{1}+ \kappa_{2}}{2})^{2} \geqslant \kappa_{1} \kappa_{2} = K$, we get from the continuity of the scalar mean curvature and the connectedness of $ \Sigma$ that either $H \leqslant 2 \lambda$ or $H \geqslant 0$. But this cannot happen since there is a point where $2 \lambda H = K > 0$ i.e. $H < 0$ and a point where $H \geqslant 0$. To see this last point, consider any plane far enough from the compact surface $\Sigma$ and move it in a fixed direction. At the first point of contact between this plane and the surface $\Sigma$, it is locally convex i.e. $\kappa_{1} \geqslant 0$ and $\kappa_{2} \geqslant 0$. We deduce that at this point $H\geqslant 0$. Therefore, $\lambda$ must be non-negative. In the same way, we prove that $H^{2} \geqslant K = 2 \lambda H$ impose that $H \geqslant 2 \lambda $ everywhere and also that $K \geqslant 4 \lambda^{2} > 0$. Hence, $\Sigma$ is an ovaloid, i.e. a compact simply-connected $C^{2}$-surface with $K > 0$, so its inner domain is a convex body \cite[Theorem 6.1]{MontielRos}.
\bigskip

Integrating the relation $2 \lambda H = K$, we get $2 \lambda \int_{\Sigma} H dA = \int_{\Sigma}K dA = 4 \pi$, the last relation coming from the Gauss Bonnet Theorem \cite[Theorem 8.38]{MontielRos}. Now, multiply the relation $2 \lambda H = K$ by the number $X \cdot \mathbf{N}(X)$, where $X$ refer to the position of any point on the surface and $\mathbf{N}$ the outer unit normal field. Integrating over $\Sigma$ and using \cite[Theorem 6.11]{MontielRos} give the following identity:
\[   A(\Sigma) = \int_{\Sigma} H X \cdot \mathbf{N}(X)dA = \dfrac{1}{2 \lambda} \int_{\Sigma} K X \cdot \mathbf{N}(X) dA = \dfrac{1}{2 \lambda} \int_{\Sigma} H dA = \frac{4\pi}{4\lambda^2}.   \]   
Consequently, we obtain $\lambda = \sqrt{\pi/ A(\Sigma)}$ and $\int_{\Sigma} H dA = \sqrt{4 \pi A(\Sigma)}$. To conclude, we apply the equality case in Minkowski inequality (\ref{inegalite_minkowski}): $\Sigma$ has to be a sphere as required. 
\end{proof}

\begin{remark}
In the proof of Proposition \ref{coro_sphere_point_critique}, we show that when the Gaussian curvature $K$ and the mean curvature $H$ are
proportional, the surface has to be a sphere. We only need $C^2$-regularity for this part. This result can be seen as a particular case of Alexandrov's uniqueness Theorem which
deals with the similar question where a relation involving $H$ and $K$ holds. Usually, more regularity is required, see e.g. \cite[Exercise 3.50]{MontielRos} and 
\cite[Appendix]{HartmanWintner}.
\end{remark}

\section{Proof of Theorem \ref{thm_min_int_absH}}
\label{sec_thm_absH}
We consider here any axisymmetric $C^{1,1}$-surface $\Sigma \in \mathcal{A}_{1,1}$ generated by an admissible Lipschitz continuous map $\theta: [0,L] \rightarrow \mathbb{R}$, where $L>0$ refers to the total length of the generating curve. We refer to Section \ref{sec_not} for precise definitions. The idea is to use again a certain rearrangement of $\theta$:
\[  \forall s \in [0,L], ~~ \theta^{\bigstar}(s) = \left\lbrace  \begin{array}{lll} \theta(s) - 2 k \pi & \mathrm{if} & \theta(s) \in [2 k \pi, (2k+1) \pi[, \quad k \in \mathbb{Z} \\
& & \\
2 k \pi - \theta(s) & \mathrm{if} & \theta(s) \in [(2k-1)\pi, 2 k \pi[, \quad k \in \mathbb{Z}. \\
\end{array} \right. \] 

As shown in Figure \ref{another-rearrangement}, it consists in reflecting all parts of the range of $\theta$ which are outside the interval $[0,\pi]$ inside it. From a geometrical point of view, it is like unfolding the surface to make it inner-convex in any direction orthogonal to the axis of revolution.
\begin{figure}[ht!]
\centering
\includegraphics{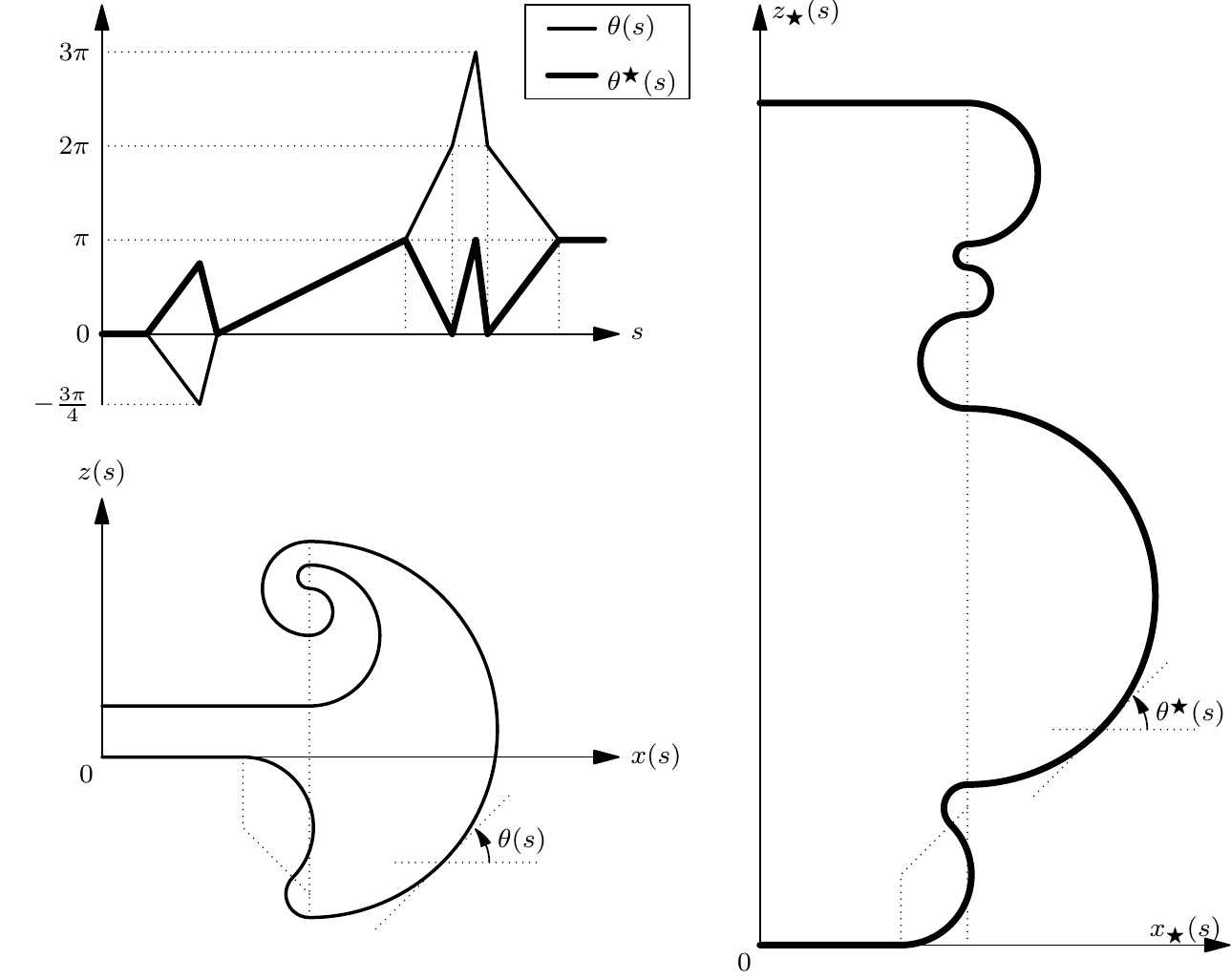}\caption{the rearrangement $\theta \mapsto \theta^\bigstar$ and the corresponding axisymmetric surfaces.}
\label{another-rearrangement}
\end{figure}

As in the proof of Theorem \ref{thm_inequality_axiconvex}, this one is divided into three steps:
\begin{enumerate}
\item We show that $\theta^{\bigstar}$ is generating an axiconvex $C^{1,1}$-surface $\Sigma^{\bigstar} \in \mathcal{A}_{1,1}^{+}$.
\item We establish that:
\[ \int_{\Sigma}\vert H \vert dA \geqslant \int_{\Sigma^{\bigstar}} H dA \geqslant \sqrt{4 \pi A(\Sigma^{\bigstar})} = \sqrt{4 \pi A(\Sigma)}. \]
\item We study the equality case.
\end{enumerate}   

\begin{proof}[Proof of Theorem \ref{thm_min_int_absH}]
\textbf{Step one:} $\Sigma^\bigstar \in \mathcal{A}_{1,1}^{+}$.
\bigskip

The map $\theta^\bigstar$ is Lipschitz continuous and valued in $[0, \pi]$ by construction. From Proposition \ref{condition_iii}, we have to check Relations \eqref{eq:01}, \eqref{eq:02} and \eqref{eq:03}. The first one comes directly from the definition of $\theta^\bigstar$. The second and third ones come from the odd and even parity of the cosine and sine functions. Indeed, observe that:
\[ \forall s \in [0,L], ~~ \left\lbrace \begin{array}{l}  
\displaystyle{ x_{\bigstar}(s) = \int_{0}^{s} \cos\theta^\bigstar(t)dt = \int_{0}^{s} \cos\theta(t)dt = x(s) } \\
\\
\displaystyle{ z_{\bigstar}(s) = \int_{0}^{s} \sin\theta^\bigstar(t)dt = \int_{0}^{s} \vert \sin\theta(t) \vert dt \geqslant z(s).} \\
\end{array} \right. \] 
Hence, we have $z_{\bigstar}(L) \geqslant z(L) > 0$, $x_{\bigstar}(L) = x(L) = 0$, and $x_{\bigstar}(s) = x(s) > 0$ for any $s \in ]0,L[$.
\bigskip

\noindent \textbf{Step 2:} comparing the total mean curvature and the area of $\Sigma$ and $\Sigma^\bigstar$.
\bigskip

Concerning the area, the equality is straightforward:
\[ A(\Sigma^\bigstar) = 2 \pi \int_{0}^{L} x_{\bigstar}(s) ds = 2 \pi \int_{0}^{L} x(s) ds = A(\Sigma). \]
Then, we have:
\[  \forall s \in [0,L], ~~ \sin\theta^\bigstar(s) + \dot{\theta}^\bigstar(s) x_{\bigstar}(s) = \left\lbrace  \begin{array}{ll} \sin\theta(s) + \dot{\theta}(s) x(s) & \mathrm{if} ~ \theta(s) \in [2 k \pi, (2k+1) \pi[,  k \in \mathbb{Z} \\
&  \\
-\sin\theta(s) - \dot{\theta}(s) x(s) & \mathrm{if} ~ \theta(s) \in [(2k-1)\pi, 2 k \pi[,  k \in \mathbb{Z}. \\
\end{array} \right. \] 
Consequently, we deduce that:
\[ \begin{array}{rcl} 
\displaystyle{ \int_{\Sigma} \vert H \vert dA} & = & \displaystyle{ \pi \int_{0}^{L} \vert \sin\theta(s) + \dot{\theta}(s) x(s) \vert ds \geqslant \pi \int_{0}^{L} \left( \sin\theta^\bigstar(s) + \dot{\theta}^{\bigstar}(s) x_{\bigstar}(s) \right) ds} \\
& & \\
& \geqslant & \displaystyle{ \int_{\Sigma_\bigstar} H dA \geqslant \sqrt{4 \pi A(\Sigma^\bigstar)} = \sqrt{4 \pi A(\Sigma)},} \\
\end{array} \]
where the last inequality comes from Theorem \ref{thm_inequality_axiconvex} applied to the axiconvex $C^{1,1}$-surface $\Sigma^\bigstar$.
\bigskip

\noindent \textbf{Step 3:} the equality case.
\bigskip

If we have equality in the above relation, it means that $\Sigma^\bigstar$ is a sphere from the equality case of Theorem \ref{thm_inequality_axiconvex}. Therefore, we have: $\theta^\bigstar(s) = \frac{\pi}{L}s$. We prove by contradiction that $\theta $ is valued in $ [0, \pi]$ which ensures from definition that $\theta = \theta^\bigstar$ i.e. $\Sigma$ is a sphere. Assume that there exists $s_{0} \in ]0,L[$ such that $\theta(s_{0}) < 0$. From the continuity of $\theta$ and the boundary conditions $\theta(0) = 0$, there exists $s_{1} \in ]0,L[$ such that $\theta(s_{1}) \in ]-\pi,0[$. Then, from the definition of $\theta^\bigstar$, $\theta(s_1)=-\theta^\bigstar(s_1)=-\frac{\pi}{L} s_1$ and by the Lipschitz continuity of $\theta$, we have:
\[ \dfrac{\theta(L)-\theta(s_{1})}{L-s_{1}} = \dfrac{\pi}{L} \dfrac{L+s_{1}}{L-s_{1}} \leqslant \Vert \dot{\theta} \Vert_{L^{\infty}(0,L)} = \Vert \dot{\theta}^\bigstar \Vert_{L^{\infty}(0,L)} = \dfrac{\pi }{L},   \]
Hence, the above inequality gives $L + s_{1} \leqslant L-s_{1}$ which is not possible since $s_{1}>0$. Let us now assume that there exists $s_{0} \in ]0,L[$ such that $\theta(s_{0}) > \pi$. More precisely, since $\theta(0)=0$, let us consider the first point $s_2\in ]0,L[$ such that $\theta(s_2)=\pi$. Since $0\leq \theta(s)<\pi$ for any $s<s_2$; we have by definition $\theta(s)=\theta^\bigstar(s)=\frac{\pi}{L} s$ for any $s<s_2$. But, passing to the limit $s \rightarrow s_2$, this leads to $\theta^\bigstar(s_2)=\pi \Leftrightarrow s_{2} = L$, which is not possible. To conclude, we proved that $\theta $ is valued in $[0,\pi]$. Hence, we have $\theta^\bigstar = \theta$ so $\Sigma$ must be a sphere. Conversely, any sphere satisfies the equality in \eqref{inegalite_minkowski}, which concludes the proof of Theorem \ref{thm_min_int_absH}.
\end{proof}

\section{Appendix: a proof of Minkowski's Theorem in the axisymmetric case}
In this section we give a short proof, {inspired by} Bonnesen \cite[Section 6,\S 35 (74)]{Bonnesen}, of Minkowski's Theorem in the axisymmetric case. This result is used
in particular in the proof
of Theorem \ref{thm_inequality_axiconvex}.
\begin{proposition}[\textbf{Bonnesen \cite{Bonnesen}}]
\label{minkowski_axi_bonnesen}
Consider any axisymmetric $C^{1,1}$-surface $\Sigma$ whose inner domain is assumed to be a convex subset of $\mathbb{R}^{3}$. Then, we have:
\[ 4 \pi \lambda^{2} - 2 \lambda \int_{\Sigma} H dA + A(\Sigma) \leqslant 0, \]
where $L = \pi \lambda $ refers to the total length of the generating curve.
\end{proposition}

\begin{proof}
{Let $\Sigma \in \mathcal{A}_{1,1} $ and $\lambda \in \mathbb{R}$ be given. Using
section \ref{sec_not}, we have in terms of generating map $\theta: [0,L] \rightarrow \mathbb{R}$:
\[ \begin{array}{rcl}
\displaystyle{ 2 \lambda^{2} -  \dfrac{\lambda}{\pi} \int_{\Sigma} H dA + \dfrac{A(\Sigma)}{2 \pi} }& = & \displaystyle{  \int_{0}^{L} \left[ \lambda^{2} \dot{\theta}(s) \sin \theta(s) ds - \lambda \left( \sin \theta(s)  + \dot{\theta}(s) x(s) \right) + x(s) \right] ds } \\
& & \\
& = & \displaystyle{ \int_{0}^{L} \left( \lambda  \sin \theta(s) - x(s) \right) \left( \lambda \dot{\theta}(s) - 1 \right) ds. } 
\end{array} \]
We perform two integration by parts and we get:
\[ \begin{array}{rcl}
\displaystyle{ 2 \lambda^{2} -  \dfrac{\lambda}{\pi} \int_{\Sigma} H dA + \dfrac{A(\Sigma)}{2 \pi} }& = &  \displaystyle{  - \int_{0}^{L} \cos \theta(s) \left( \lambda \theta(s) - s \right) \left( \lambda \dot{\theta}(s) - 1 \right)ds   } \\ 
& & \\
& = &\displaystyle{  \dfrac{1}{2} \left( \lambda \pi - L  \right)^{2} - \dfrac{1}{2}\int_{0}^{L} \left( \lambda \theta(s) - s \right)^{2} \dot{\theta}(s) \sin\theta(s) ds .   }  \\
\end{array} \]}
Now we set $ \lambda = \frac{1}{\pi} L$ and we assume that $\Sigma$ is inner-convex and axisymmetric. Therefore, the Gaussian curvature $K(s) = \kappa_{1}(s) \kappa_{2}(s) = \dot{\theta}(s) \frac{\sin\theta(s)}{x(s)}$ is non-negative on $[0,L]$. Hence, we obtain the required inequality:
\[  4 \pi \lambda^{2} - 2 \lambda \int_{\Sigma} H dA + A(\Sigma) =  - \pi \int_{0}^{L} \left( \lambda \theta(s) - s \right)^{2} K(s) x(s) ds \leqslant 0,  \]
which concludes the proof of Proposition \ref{minkowski_axi_bonnesen}.
\end{proof}

\begin{corollary}
\label{coro_minkowski_axi_bonnesen}
Consider the class $\mathcal{C}$ of axisymmetric inner-convex $C^{1,1}$-surfaces. Then, we have the following inequality:
\[ \forall \Sigma \in \mathcal{C}, \quad \int_{\Sigma} H dA \geqslant \sqrt{4 \pi A(\Sigma)}, \]
where the equality holds if and only if $\Sigma$ is a sphere.
\end{corollary}

\begin{proof}

{From Proposition \ref{minkowski_axi_bonnesen}, the polynomial in $\lambda$ has real roots, thus its discriminant must be non-negative, which gives the above inequality. }

{Now if equality holds, $\lambda=L/\pi$ is a double root, that is:}
\[ \int_{0}^{L} (\lambda \theta (s) - s)^{2}  \sin\theta(s) \dot{\theta}(s)  ds = 0  \]
Hence, the integrand must be zero almost everywhere, i.e. $\dot{\theta} $ is equal to zero or to $\frac{1}{\lambda}$ a.e. on $[0,L]$. But we have:
\[ \int_{0}^{L} \dot{\theta}(s)ds = \theta(L) - \theta(0) = \pi = \frac{1}{\lambda} \vert \lbrace s \in [0,L], ~~ \dot{\theta}(s) \neq0 \rbrace \vert \]
Since $ \pi \lambda = L $, we get that $\dot{\theta}(s) \neq 0$ almost everywhere thus $\dot{\theta} = \frac{1}{\lambda}$ a.e. Hence, we get that $\theta$ is linear everywhere since the constant function $\frac{1}{\lambda}$ is continuous and $\Sigma$ is a sphere as required.
\end{proof}

\section*{Acknowledgement}

The work of J\'er\'emy Dalphin, Antoine Henrot and Tak\'eo Takahashi is supported by the project ANR-12-BS01-0007-01-OPTIFORM {\it Optimisation de formes} financed by the French Agence Nationale de la Recherche (ANR).

\bibliographystyle{plain}
\bibliography{biblio}
\end{document}